\documentclass[a4paper]{amsart} % sets paper format is a4
\usepackage[ascii]{inputenc} % for compatibility, please do not use utf8, let alone latin1
\usepackage{amssymb}
\usepackage{mathrsfs}

\usepackage{xcolor}

\def\hook{\upharpoonright}
\def\forces{\Vdash}
\def\P{\mathbb P}

\usepackage[shortlabels]{enumitem}
\setlist[enumerate,1]{label={(\Alph*)}}
\setlist[enumerate,2]{label={(\alph*)}}
\setlist[enumerate,3]{label={$\bullet_{\arabic*}$}}

\newtheorem{theorem}{Theorem}[section]

\newtheorem{claim}[theorem]{Claim}

\newtheorem{lemma}[theorem]{Lemma}

\newtheorem{proposition}[theorem]{Proposition}

\theoremstyle{definition}

\newtheorem{definition}[theorem]{Definition}

\theoremstyle{remark}

\newtheorem{question}[theorem]{Question}
\newtheorem{remark}[theorem]{Remark}

\newcommand{\cf}{\mathrm{cf}}

\newcommand{\CH}{\mathrm{CH}}

\newcommand{\dom}{\mathrm{dom}}

\newcommand{\GCH}{\mathrm{GCH}}

\newcount\skewfactor
\def\mathunderaccent#1#2 {\let\theaccent#1\skewfactor#2
\mathpalette\putaccentunder}
\def\putaccentunder#1#2{\oalign{$#1#2$\crcr\hidewidth
\vbox to.2ex{\hbox{$#1\skew\skewfactor\theaccent{}$}\vss}\hidewidth}}

\newbox\noforkbox \newdimen\forklinewidth
\setbox0\hbox{$\textstyle\bigcup$}
\setbox1\hbox to \wd0{\hfil\vrule width \forklinewidth depth \dp0
                        height \ht0 \hfil}
\setbox\noforkbox\hbox{\box1\box0\relax}
\def\unionstick{\mathop{\copy\noforkbox}\limits}
\def\nonfork#1#2_#3{#1\unionstick_{\textstyle #3}#2}
\def\nonforkin#1#2_#3^#4{#1\unionstick_{\textstyle #3}^{\textstyle
    #4}#2}
\setbox0\hbox{$\textstyle\bigcup$}
\setbox1\hbox to \wd0{\hfil{\sl /\/}\hfil}
\setbox2\hbox to \wd0{\hfil\vrule height \ht0 depth \dp0 width
                                \forklinewidth\hfil}
\newbox\doesforkbox
\setbox\doesforkbox\hbox{\box1\box0\relax}
\def\nunionstick{\mathop{\copy\doesforkbox}\limits}

\def\fork#1#2_#3{#1\nunionstick_{\textstyle #3}#2}
\def\forkin#1#2_#3^#4{#1\nunionstick_{\textstyle #3}^{\textstyle
    #4}#2}

\newcommand{\stickT}{%
\setbox255=\hbox{\raise1ex\hbox{$\hspace{0.2pt}\,\bullet\,$}}
\mathord{\rlap{\hbox to\wd255{\hss\hbox{$|$}\hss}}
\box255}
}
\newcommand{\stickS}{%
\setbox255=\hbox{\raise0.6ex\hbox{$\scriptstyle\bullet$}}
\mathord{\rlap{\hbox to\wd255{\hss\hbox{$\scriptstyle|$}\hss}}
\box255}
}

\usepackage{hyperref}

\author[P. Marun]{Pedro Marun}
\address[P. Marun]{
Institute of Mathematics, 
Czech Academy of Sciences, 
{\v Z}itn{\'a} 25, Prague 1, 
115 67, Czech Republic
}
\urladdr{https://pedromarun.github.io}

\author[S. Shelah]{Saharon Shelah}
\address[S. Shelah]{Einstein Institute of Mathematics,
The Hebrew University of Jerusalem,
9190401, Jerusalem, Israel; and\\
Department of Mathematics,
Rutgers University,
Piscataway, NJ 08854-8019, USA}
\urladdr{https://shelah.logic.at/}
\author[C. B. Switzer]{Corey Bacal Switzer}
\address[C. B. Switzer]{Kurt G\"{o}del Research Center, Faculty of Mathematics, University of Vienna, Kolingasse 14 -- 16, 1090 Wien, Austria}
\urladdr{}
\thanks{First typed 2025-11-24.
The second author is grateful to Craig Falls for generously funding typing services that were used during the work on the paper. Research of the second author was partially supported by the Israel Science Foundation (ISF) grant no: 2320/24 (2023-2027) and by the grant ``Independent Theories'' NSF-BSF, (BSF 3013005232). The first author was supported by the Czech Academy of Sciences (RVO 67985840). The research of the third author was funded in whole or in part by the Austrian Science Fund (FWF) grant doi 10.55776/ESP548. For open access purposes, the author has applied a CC BY public copyright license to any author-accepted manuscript version arising from this submission. References like [Sh:950, Th0.2=Ly5] mean that the internal label of Th0.2 is y5 in Sh:950.
The reader should note that the version on S. Shelah's website is usually more up-to-date than the one in arXiv.
This is publication number    
1270
in Saharon Shelah's list.
}

\makeatletter
\@namedef{subjclassname@2020}{\textup{2020} Mathematics Subject Classification}
\makeatother
\date{November 24th, 2025} % add date of last edit here, do not use \today
 
\renewcommand{\epsilon}{\varepsilon}
\newcommand{\restr}{\mathord{\upharpoonright}}
\newcommand{\seq}[1]{{\langle{#1}\rangle}}
\newcommand{\F}{\mathcal{F}}
\newcommand{\Add}{\mathrm{Add}}
\newcommand{\B}{\mathbb{B}}
\newcommand{\C}{\mathbb{C}}
\newcommand{\ro}{\mathrm{ro}}
\newcommand{\A}{\mathbb{A}}
\usepackage{changes}
\reversemarginpar

\definechangesauthor[name={Pedro Marun}, color=teal]{PM}
\definechangesauthor[name={Corey Switzer}, color=orange]{CS}

\title{Adding $\aleph_\omega$ Many Cohen Reals}

\begin{document}
\begin{abstract}
    Abstractly, the generic extensions after $\aleph_\omega$-many Cohen reals and $\aleph_{\omega+1}$-many Cohen reals must be different for reasons of uniform density the relevant Boolean algebras. Nevertheless this is not satisfying and it would be nice to pin the difference between the two models down to some mathematical or combinatorial principle. In this paper we provide such a principle. 
\end{abstract}
\maketitle

\setcounter{section}{-1}
\section{Introduction}
%Hello. Test-citations: \cite{Sh:a} \cite{Sh:19} \cite{Sh:377}
%\begin{enumerate}
%    \item Level 1
%    \begin{enumerate}
%        \item Level 2
%        \begin{enumerate}
%            \item Level3
%        \end{enumerate}
%    \end{enumerate}
%\end{enumerate}

Let $\Add(\omega,\aleph_\omega)$ be the poset for adding $\aleph_{\omega}$-many Cohen reals . In $V^\P$, the continuum must be at least $\aleph_{\omega}$. By K\"onig's Lemma, the continuum has uncountable cofinality, hence it must be at least $\aleph_{\omega+1}$. One can then ask whether this is different from adding $\aleph_{\omega+1}$-many Cohen reals, i.e. forcing with $\Add(\omega,\aleph_{\omega+1})$. More precisely, if $G$ is $\Add(\omega,\aleph_\omega)$-generic over $V$, is there an $H$ which is $\Add(\omega,\aleph_{\omega+1})$-generic over $V$ such that $V[G]=V[H]$. The (folklore) answer to this is ``No", and in fact $V[G]$ doesn't even contain a filter which is $\Add(\omega,\aleph_{\omega+1})$-generic over $V$,  Proposition \ref{density prop} below. However, the proof of this is somewhat unsatisfying, as it relies on abstract forcing considerations involving the (uniform) densities of the complete Boolean algebras. It is therefore desirable to exhibit a combinatorial or ``mathematical" principle which distinguishes the two models. The purpose of this note is to provide one.

%The generic extensions by adding $\aleph_\omega$ Cohen reals and adding $\aleph_{\omega + 1}$ Cohen reals must be different simply because the uniform density of the Boolean completions is distinct\footnote{This is elaborated on below, see Proposition \ref{density prop}.} hence one cannot completely embed the algebra for adding $\aleph_{\omega+1}$ many Cohen reals into the algebra for adding $\aleph_\omega$ many - even when restricting to dense sets below some condition on each side. This is not satisfying however as one would like a combinatorial or mathematical principle which separates the two models. The purpose of this note is to provide one.

\begin{definition} \label{maindef}
    Let $\sigma \leq \theta < \mu < \lambda$ be cardinals. The statement $\mathsf{Pr}(\sigma, \theta, \mu, \lambda)$ states that there is a $\overline{\F}$ satisfying the following:

    \begin{enumerate}
        \item $\overline{\F} = \langle \F_c : c \in [\lambda]^\sigma\rangle$ is a sequence so that for all $c \in [\lambda]^\sigma$ we have $\F_c \subseteq \mathcal{P}(c)$.
        \item For all $c \in [\lambda]^\sigma$ we have that $\F_c$ has cardinality at most $\theta$.
        \item For every $A \in [\lambda]^\lambda$ there is a sequence $ \langle B_i : i < \mu\rangle$ so that 
        \begin{enumerate}
            \item $A = \bigcup_{i<\mu}B_i$ and
            \item $B_i \cap c \in \F_c$ for all $c \in [\lambda]^\sigma$ and $i < \mu$.
        \end{enumerate}
    \end{enumerate}
\end{definition}

We will show the following.
\begin{theorem} \label{main}
    Assume $\GCH$. Let $\P_0=\Add(\omega,\aleph_\omega)$ and let $\P_1=\Add(\omega,\aleph_{\omega+1})$. The following hold:

    \begin{enumerate}
        \item For all $n < \omega$ $\forces_{\P_0} \mathsf{Pr}(\aleph_n, \aleph_{n+1}, \aleph_\omega, \aleph_{\omega+1})$
        \item For all $n < \omega$ $\forces_{\P_1} \neg \mathsf{Pr}(\aleph_n, \aleph_{n+1}, \aleph_\omega, \aleph_{\omega+1})$
    \end{enumerate}
\end{theorem}
Thus the above says that in particular $\mathsf{Pr}(\aleph_0, \aleph_1, \aleph_\omega, \aleph_{\omega+1})$ is the desired principle. In fact Theorem \ref{main} holds in more generality - namely we can take $\mu$ to be a strong limit cardinal of countable cofinality, and $\sigma < \mu$ infinite and we get that adding $\mu$ many Cohen reals forces $\mathsf{Pr}(\sigma, 2^\sigma, \mu, \mu^+)$ while adding $\mu^+$-many Cohen reals forces $\neg \mathsf{Pr}(\sigma, 2^\sigma, \mu, \mu^+)$. These two pieces are proved below as Lemmas \ref{addingfew} and \ref{addingmany} respectively.

The rest of this paper is organized as follows. In the next section we make some preliminary observations about adding $\aleph_\omega$ versus adding $\aleph_{\omega+1}$ Cohen reals. In Section \ref{Sectionfew} we prove the generalization of Item (1) of Theorem \ref{main} while in Section \ref{sectionmany} we prove the generalization of Item (2) of Theorem \ref{main}. The paper concludes with some questions for further research. Throughout our notation is mostly standard, conforming e.g. to that of \cite{jechSetTheory2003} or \cite{kunenSetTheory2011}. Those texts are referred to for any undefined terms. 

\section{Some Preliminary Observations}

If $\kappa$ is an infinite cardinal, we let $\Add(\omega,\kappa)$ be the poset for adding $\kappa$-many Cohen reals. More precisely, $\Add(\omega,\kappa)$ is the set of finite partial functions $\kappa \rightharpoonup 2$, ordered by reverse inclusion. We will argue that, if $\kappa<\lambda$, then a generic extension of $V$ by $\Add(\omega,\kappa)$ does not contain a filter that is $\Add(\omega,\lambda)$-generic over $V$. This result is folklore but a proof is hard to find in the published literature, we will follow the outline in \cite{MOanswer}.

We collect some facts from the elementary theory of Boolean algebras. If $\B$ is a Boolean algebra and $b\in\B$, then $\B_b=\{c\in \B:c\le b\}$, which can be made into a Boolean algebra (with top element $b$) in an obvious way. We let $\B^+=\B\setminus\{0\}$. If $\A$ is a subalgebra of $\B$, then $\A$ is said to be a \emph{regular subalgebra of $\B$} iff for all $X\subseteq \A$, if $\bigvee^\A X$ exists, then so does $\bigvee^\B X$ and moreover $\bigvee^\A X=\bigvee^\B A$. If $\B$ is complete, we will say that $\A\subseteq\B$ is a \emph{complete subalgebra} if it is a regular subalgebra of $\B$ that is also complete as a Boolean algebra.

We will need the following cardinal function:

\begin{definition}
If $\B$ is a Boolean algebra, we define its \emph{density} by
\[
d(\B):=\min\{|D|: D \text{ is dense in }\B\},
\]
where $D$ is \emph{dense} in $\B$ iff $D\subseteq \B^+$ and for all $b\in \B^+$ there exists $d\in D$ with $d\le b$.

We say that $\B$ has \emph{uniform density} iff $d(\B)=d(\B_b)$ for all $b\in\B^+$.
\end{definition}

Unlike the topological notion, the Boolean algebraic density is nicely monotone:

\begin{lemma}[folklore?]\label{densitySub}
Let $\A$ be a complete subalgebra of a complete Boolean algebra $\B$. Then $d(\A)\le d(\B)$.
\end{lemma}

\begin{proof}
Let $D\subseteq \B$ be dense. For  $d\in D$, let
\[
u_d:=\bigwedge\{a\in \A:d\le a\}.
\]
We will be done if we can argue that $\{u_d:d\in D\}$ is dense in $\A$. So, fix $a\in \A$. By the density of $D$, there exists some $d\in D$ with $d\le a$, hence $u_d\le a$.
\end{proof}

\begin{lemma}[{\cite[Lemma 25.5(a)]{jechSetTheory1978}}]\label{isoBelow}
Let $\B$ and $\C$ be complete Boolean algebras. Let $G$ and $H$ be $\B$ and $\C$ generic over $V$, respectively, and suppose that $V[G]=V[H]$. Then there exists $\pi\in V$, $b\in \B$ and $c\in\C$ such that $\pi:\B_b\cong \C_c$ and $\pi[\B_b\cap G]=\C_c\cap H$.
\end{lemma}

The \emph{Cohen algebra on $\kappa$} is $\C_\kappa:=\ro(\Add(\omega,\kappa))$, i.e. the Boolean completion of $\Add(\omega,\kappa)$.

\begin{lemma}\label{homog}
Let $\kappa$ be an infinite cardinal. The algebra $\C_\kappa$ is homogeneous and has (uniform) density $\kappa$.
\end{lemma}

\begin{proof}
    For the homogeneity, see \cite[Corollary 12.5]{koppelbergHandbookBooleanAlgebras1989}. The uniform density assertion follows from homogeneity and the fact that $d(\C_\kappa)$ is also the smallest size of a dense subset of $\Add(\omega,\kappa)$, which is easily seen to be $\kappa$.
\end{proof}

%We begin by spelling out the point made in the first sentence of this paper.\comment[id=PM]{This might make more sense in a different section}
\begin{proposition}[folklore?] \label{density prop}
Let $\kappa<\lambda$ be infinite cardinals. In $V^{\Add(\omega,\kappa)}$ there is no $\Add(\omega,\lambda)$-filter that is generic over $V$.
\end{proposition}

\begin{proof}
Suppose not. Then we can find $G$ which is $\C_\kappa$-generic over $V$ and $H$ which is $\C_\lambda$-generic over $V$ such that $H\in V[G]$. Obviously, $V\subseteq V[H]\subseteq V[G]$. By the intermediate model theorem (see \cite[Lemma 15.43]{jechSetTheory2003}, there exists $\B$ a complete regular subalgebra of $\C_\kappa$ such that $V[H]=V[G\cap \B]$. By Lemma \ref{isoBelow}, there exist $b\in \B$ and $c\in \C_\lambda$ such that $\B_b \cong (\C_\lambda)_c$ (in $V$). But then, by Lemmas \ref{densitySub} and \ref{homog},
\[
\lambda = d((\C_\lambda)_c)=d(\B_b)\le d(\B)=\kappa,
\]
which is a contradiction.
\end{proof}

\begin{remark}
If we don't insist on fixing the ground model, the previous theorem can fail. In other words, it is consistent (modulo large cardinals) that there is a pair of models $V\subseteq W$ such that adding few Cohen reals to $W$ adds a lot of Cohen reals to $V$, see \cite{gitikAddingLotCohen2015} and \cite{gitikAddingLotCohen2015a}.
\end{remark}

Towards a slightly more mathematical distinction between the two models we also observe the following.

\begin{proposition}
Assume $\CH$. Let $G_0$ be $\Add(\omega,\aleph_\omega)$-generic over $V$ and let $G_1$ be $\Add(\omega,\aleph_{\omega+1})$-generic over $V$.
%Assume $\CH$. Let $\P_0$ be the forcing notion for adding $\aleph_\omega$-many Cohen reals and $\P_1$ be the forcing notion for adding $\aleph_{\omega+1}$-many. Let $G_0$ be $\P_0$-generic over $V$ and let $G_1$ be $\P_1$-generic over $V$. 
\begin{enumerate}
\item
In $V[G_0]$ there are $\aleph_1$-many Borel functions $\{f_\alpha : \alpha \in \omega_1\}$ and a set of $\aleph_\omega$-many reals, $A = \{a_\alpha : \alpha \in \aleph_\omega\}$ so that the following hold.
\begin{enumerate}
\item For all $\alpha < \omega_1$ the domain of $f_\alpha$ is $[\omega^\omega]^\omega$.
\item For every $x \in \omega^\omega$ there is a $c \in [\aleph_\omega]^\omega \cap V$ and an $\alpha < \omega_1$ so that $x = f_\alpha (\{a_\xi : \xi \in c\})$.
\end{enumerate}

\item In $V[G_1]$ the converse of the above holds, namely, for every $\aleph_1$-many Borel functions $\{f_\alpha : \alpha \in \omega_1\}$ and every set of $\aleph_\omega$-many reals, $A = \{a_\alpha : \alpha \in \aleph_\omega\}$ there is an $x \in \omega^\omega$ so that $x \notin \{f_\alpha (\{a_\xi : \xi \in c\}) : \alpha \in \omega_1\}$ for any $c \in [\aleph_\omega]^\omega \cap V$. 

\end{enumerate}
\end{proposition}

We note in contrast to this there is no concrete, cardinal characterstic that could separate the models $V[G_0]$ and $V[G_1]$ as these must in fact be the same, see \cite[Section 11.3]{Bls10}.

\begin{proof}
Each part is a consequence of the ccc. For the first part let $\{f_\alpha : \alpha \in \omega_1\}$ be the ground model Borel functions and let $A = \{a_\alpha : \alpha \in \aleph_\omega\}$ be the set of Cohen generic reals coming from $G_0$. If $\dot{x}$ is a nice name for a real then it is well known, see e.g. \cite[Theorem 4.1.2]{zapletalForcingIdealized2008}, that there are $c \in [\aleph_\omega]^\omega \cap V$ so that $\dot{x}$ is forced to be in $V[a_\xi : \xi \in c]$ and in that model there is an $\alpha < \omega_1$ so that $\dot{x}^{G_0} = f_\alpha (\{a_\xi : \xi \in c\})$. For the second part observe that any $\aleph_\omega$-many reals and any $\aleph_1$-many Borel functions were added by $\aleph_\omega$-many of the Cohen reals in $G_1$ and hence there is a Cohen real generic over the model containing all of them. 
\end{proof}

We conclude this section by making some simple observations about the principle $\mathsf{Pr}(\sigma, \theta, \mu, \lambda)$ in an attempt to clarify it and prepare the reader for the arguments in Lemmas \ref{addingfew} and \ref{addingmany}. All of these are easy, as should be obvious to the reader who is clear on what the principle states.
\begin{proposition}
Let $\sigma \leq \theta \leq \mu \leq \lambda$ be infinite cardinals.
\begin{enumerate}
\item If $\theta = 2^\sigma$, then $\mathsf{Pr}(\sigma, \theta, \mu, \lambda)$ is true for any choice of $\mu$ and $\lambda$. 
\item If $\lambda=\mu^+$, the statement $\mathsf{Pr}(\sigma, \theta, \mu, \lambda)$ is unchanged by allowing in clause (C) of the definition that $A$ has size $\leq \lambda$.
\item $\Pr(\sigma,\theta,\mu,\lambda)$ is monotone in $\theta$ and $\mu$. More precisely, if $\theta_0 \le \theta_1$ or $\mu_0\le \mu_1$ are infinite cardinals, then $\Pr(\sigma,\theta_0,\mu_0,\lambda)$ implies $\Pr(\sigma,\theta_1,\mu_1,\lambda)$.
\end{enumerate}

\end{proposition}

\begin{proof}
For (A) note that if $\theta = 2^\sigma$ then letting $\mathcal F_c := \mathcal{P}(c)$ satisfies the requirements of the principle. 

For (B), note that since $\sigma \leq \theta$ we have that for any $\overline{\F}$ as in the statement of Definition \ref{maindef} we can define $\overline{\F}'$ by $\F'_c := \{\emptyset\}\cup \{\{\alpha\}\mid \alpha \in c\} \cup \F_c$ and $\F'_c$ will still have size at most $\theta$. Now for each $A \subseteq \lambda$ of size ${<}\lambda$, by $\lambda=\mu^+$ we can enumerate $A$ (possibly with repetitions) as $A=\{\alpha_\xi:\xi<\mu\}$. Then $\{\alpha_\xi\}\cap c \in \F'_c$ for any $c\in [\lambda]^\sigma$.

For (C), any witness to $\Pr(\sigma,\theta_0,\mu_0,\lambda)$ will witness $\Pr(\sigma,\theta_1,\mu_0,\lambda)$. For the monotonicity in $\mu$, let $\F$ witness $\Pr(\sigma,\theta_0,\mu_0,\lambda)$ and note that we may assume without loss of generality that $\emptyset\in \F_c$ for every $c\in [\lambda]^{\sigma}$. Given $A\in [\lambda]^\sigma$, let $\seq{B_\xi:\xi<\mu_0}$ witness Definition \ref{maindef}(C) and set $B_\xi=\emptyset$ for $\xi \in [\mu_0,\mu_1)$. This works.
\end{proof}

As a consequence of the above we have the following.

\begin{lemma}
If $\GCH$ holds then $\mathsf{Pr}(\sigma, \theta, \mu, \lambda)$ holds for every infinite $\sigma < \theta$. 
\end{lemma}

\section{Adding few Cohens}\label{Sectionfew}
In this section we prove the promised generalization of part 1 of Theorem \ref{main}. First, a definition:

\begin{definition}\label{def: cov}
Let $\P$ be a forcing poset and $\kappa$ an cardinal. We say that $\P$ has the \emph{$\kappa$-covering property} iff for every $p\in\P$ and every $\P$-name $\dot X$ for a set of ordinals with $p\Vdash |\dot X|<\kappa$, there exists $q\le p$ and $Y$ with $|Y|<\kappa$ and $q\Vdash \dot X\subseteq \check{Y}$.
\end{definition}

Clearly, if $\P$ has the $\kappa$-covering property, then forcing with $\P$ preserves $\kappa$. Also, if $\kappa$ is regular, then every $\kappa$-cc poset has the $\kappa$-covering property.

%\begin{lemma}[folklore]\label{cccCovering}
%Let $\kappa$ be an uncountable regular cardinal and $\P$ a poset. Then the following are equivalent:
%\begin{enumerate}
%\item $\P$ has the $\kappa^+$-cc.
%\item for every $\P$-name $\dot X$ for a set of ordinals, if $\Vdash |\dot X| \le \kappa$, then there exists a set $Y$ of ordinals such that $|Y|\le \kappa$ and $\Vdash\dot{X}\subseteq\check{Y}$.
%\end{enumerate}
%\end{lemma}
%
%\begin{proof}
%For (A) implies (B), let $\dot{X}$ be a name for a set of ordinals and assume $\forces_\P |\dot X| \leq \kappa$. Let $\{\dot x_\alpha\mid \alpha < \kappa\}$ be a name for an injective enumeration of the elements of $\dot X$. For each $\alpha < \kappa$ there is a maximal antichain of conditions $A_\alpha$ deciding $\dot{x}_\alpha$. Let $B_\alpha$ be those decisions. Since $\P$ has the $\kappa^+$-c.c. this set has size at most $\kappa$. Clearly then $Y:=\bigcup_{\alpha \in \kappa} B_\alpha$ is as desired. To see that (B) implies (A), let $W$ be a maximal antichain of $\P$ and suppose that $\lambda:=|W|\ge \kappa^+$. Consider an injective enumeration $W=\{p_\xi:\xi<\lambda\}$ and let $\dot\alpha$ be a $\P$-name such that $\Vdash \{p_{\dot\alpha}\}=\dot{G}\cap W$. Apply (B) to (a $\P$-name for) the set $\{\dot\alpha\}$ to find $Y$ with $|Y|\le \kappa$ and $\Vdash \{\dot\alpha\}\subseteq\check{Y}$. Since $\lambda \ge \kappa^+$, there is some $\beta\in \lambda\setminus Y$. But $p_\beta \Vdash \dot\alpha=\check{\beta}$, hence $p_\beta\Vdash \dot\alpha\not\in \check{Y}$, contradiction.
%\end{proof}

In the next lemma, if $\kappa$ is a cardinal and we're working in a generic extension $V^\P$, then the symbol $|\kappa|$ refers to the cardinality of $\kappa$ as computed in $V^\P$.

\begin{lemma} \label{addingfew}
    Assume $2^\sigma = \theta < \mu <\lambda$. Let $\P$ be a forcing notion that has the $\sigma^+$-covering property and such that $\Vdash |\dot G|\le |\mu|$. Then $\forces_{\P} \mathsf{Pr}(|\sigma|, |\theta|, |\mu|, |\lambda|)$.
\end{lemma}

Note that the first part of Theorem \ref{main} follows from the above by assuming $\GCH$ and subbing in $(\aleph_n, \aleph_{n+1}, \aleph_\omega, \aleph_{\omega +1})$ for $(\sigma, \theta, \mu, \lambda)$.  

\begin{proof}
    Fix $\sigma, \theta, \mu, \lambda$ and $\P$ as in the statement of the theorem. Since $\P$ has the $\sigma^+$-covering property, there is a name $\dot\Phi$ such that the following is forced by $\P$: ``$\dot\Phi:[\lambda]^{|\sigma|}\to ([\lambda]^\sigma)^V$ and for all $c\in [\lambda]^{|\sigma|}$, $c\subseteq{\dot\Phi}(c)$".
    
    %By Lemma \ref{cccCovering}, there is a name $\dot F$ so that the following is forced by $\P$: ``$\dot F:[\lambda]^\sigma \to ([\lambda]^\sigma)^V$ and for all $c\in [\lambda]^\sigma$, $c\subseteq \dot F(c)$".
    
    %Let $G\subseteq \P$ be $V$-generic but work currently in $V$. We begin by choosing a $\P$-name $\dot{F}$ so that $\forces$`` $\dot{F}$ is a function with domain $[\mu^+]^\sigma$ and range contained in $V \cap [\mu^+]^\sigma$ with the property that for all $c \in [\mu^+]^\sigma$ we have $\dot{F}(c) \supseteq c$". Let us first see why such an $\dot{F}$ exists. By the ccc we can enumerate the nice names for elements of $[\mu^+]^\sigma$ in the extension as $\{\dot{u}_\alpha : \alpha \in \mu^\sigma\}$ and for each $\alpha \in \mu^\sigma$ define its {\em outerhull}: $v_\alpha : = \{\xi < \mu^+ : \nVdash \xi \notin \dot{u}_\alpha\}$. It is routine to check that $\forces \dot{u}_\alpha \subseteq \check{v}_\alpha$ for every $\alpha < \mu^\sigma$ and $v_\alpha$ has cardinality $\sigma$ (again by the ccc). Finally - in $V[G]$ define $F$ to be the function so that $F(u) = v_\alpha$ with $\alpha$ least so that $u = \dot{u}_\alpha^G$. Back in $V$ let $\dot{F}$ be a name for this function. It is clear it has the requisite properties then.

    Working in $V^\P$, for each $c \in [\lambda]^{|\sigma|}$ let $\F_c = \{a\cap c : a \in V \cap \mathcal{P}(\dot{\Phi}(c))\}$. Note that for every $c$ we have that $\F_c$ has cardinality at most $(2^\sigma)^V = |\theta|$ and hence $\overline{\F} = \langle \F_c : c \in [\lambda]^\sigma\rangle$ satisfies the first two clauses of Definition \ref{maindef}. Back in $V$ let $\dot{\overline{\F}}$ name this sequence and for each $\dot{c}$ let $\dot{\F}_{\dot{c}}$ be the name for the associated element. We need to see that the third clause of Definition \ref{maindef} is forced. 
    %Suppose therefore that there is a condition $p \in \P$ so that $p\forces \dot{A} \in [\lambda]^\lambda$. For each $\alpha < \lambda$ choose a condition $p_\alpha \leq p$ and an ordinal $\xi_\alpha<\lambda$ so that $p_\alpha$ forces $\xi_\alpha$ to be the $\alpha^{\rm th}$ element of $\dot{A}$. Since the forcing has size $\mu$ there is a $p^* \leq p$ which is equal to $p_\alpha$ for $\lambda$-many $\alpha$. Now for each $q \leq p^*$ let $B_{\dot{A}, q} = \{\xi : q \forces \xi \in \dot{A}\}$. Obviously $p^* \forces \dot{A} = \bigcup_{q \in \dot{G}} B_{\dot{A}, q}$. 

    %To finish now note that for any $q \leq p^*$ we have $B_{\dot{A}, q} \cap c = B_{\dot{A}, q} \cap\dot{F}(c)\cap c$ and $B_{\dot{A}, q} \cap \dot{F}(c) \in V$, as needed. \\

    %\comment[id=PM]{Am I insane?}\comment[]{I think this is correct. Let's maybe mention the specific case we want just to make it clear. }%
    Suppose that $\dot{A}$ is a $\P$-name such that $\Vdash \dot A\subseteq \lambda$. For each $q\in\P$, let $B_q=\{\xi<\lambda:q\Vdash \xi\in\dot{A}\}$. Note that $\Vdash \dot{A}=\bigcup_{q\in \dot{G}}B_q$. Now let $\dot c$ be a $\P$-name for a $|\sigma|$-sized subset of $\lambda$. If $q\in \P$, then $\Vdash B_q\cap \dot{c}=B_q\cap \dot{\Phi}(\dot c)\cap \dot{c}\in \dot{\F}_{\dot c}$ because $\Vdash \dot{\Phi}(\dot c)\in V$.
\end{proof} 

\begin{remark}
The proof of the previous lemma shows something slightly more general, because it produces a decomposition for each subset of $\lambda$, not just those in $[|\lambda|]^{|\lambda|}$. Also, the decomposition is of size at most $|\mu|$, not exactly $\mu$ (to get exactly $\mu$, simply pad with $\emptyset$'s).
\end{remark}

\section{Adding Many Cohens}\label{sectionmany}

Now we prove the failure of the principle after adding $\mu^+$-many Cohen reals thus complementing Lemma \ref{addingfew}.

\begin{lemma} \label{addingmany}
    Assume $\mu$ is a strong limit cardinal of countable cofinality, $\sigma < \theta < \mu$ and $\lambda= \mu^+$. If $\P=\Add(\omega,\lambda)$, then $\forces_{\P} \neg \mathsf{Pr}(\sigma, \theta, \mu, \lambda)$.
    %Assume $\mu$ is a strong limit cardinal of countable cofinality, $2^\sigma = \theta < \mu$ and $\lambda= \mu^+$. If $\P=\Add(\omega,\lambda)$, then $\forces_{\P} \neg \mathsf{Pr}(\sigma, \theta, \mu, \lambda)$.
\end{lemma}

\begin{proof}
  %Let $G \subseteq \P$ be generic over $V$ but work for the moment in $V$ still. Again let $\sigma, \theta, \mu, \lambda$ be as in the statement of the lemma. Enumerate moreover the names for the generic Cohen reals by $\{\dot{\eta}_\alpha \; | \; \alpha \in \lambda\}$ where each $\dot{\eta}_\alpha$ is forced to be in $2^\omega$. We want to show that $\mathsf{Pr}(\sigma, \theta, \mu, \lambda)$ fails in $V[G]$ so suppose towards a contradiction that $\dot{\overline{\F}}$ is a name which is forced by the empty condition to be as in Definition \ref{maindef}. We need to produce a name for a subset of $[\lambda]^\lambda$ which cannot be decomposed as in the statement of Definition \ref{maindef}. Actually we will produce $\theta^+$ many sets and show by contradiction they cannot all be decomposed as we want.

  We argue by contradiction. By homogeneity of $\P$, we may assume that $\Vdash_\P \Pr(\sigma,\theta,\mu,\lambda)$. Let $\dot{\bar{\F}}$ witness this. We aim to produce a name for a $\lambda$-sized subset of $\lambda$ which cannot be decomposed as in Definition \ref{maindef}. Actually, we will produce $\theta^+$ many sets and show by contradiction they cannot all be decomposed as we want.

  Let $\dot g$ be a $\P$-name for the generic function $\lambda\to 2$. For $\xi < \theta^+$, let $\dot{A}_\xi$ be a name for the set $\{\alpha <\lambda:  \dot{g}(\theta^+\cdot\alpha + \xi)= 1\}$. An easy genericity argument confirms that each $\dot{A}_\xi$ is forced to have has size $\lambda$ and that they are all distinct. 
  
  Since we assume it is forced that $\mathsf{Pr}(\sigma, \theta, \mu, \lambda)$ holds we have names $\dot{B}_{\xi, i}$ for $\xi < \theta^+$ and $i < \mu$ so that
  \[
  \forces \dot{A}_\xi = \bigcup \{\dot{B}_{\xi, i} : i \in \mu\}
  \]
  Definition \ref{maindef}(C) holds. Our goal is to find a $c \in [\lambda]^\sigma$ such that $c\cap B_{\xi, i}$ are pairwise distinct for $\theta^+$ many pairs $(\xi, i)$. Since each of these must, by definition, be in $\dot{\overline{\F}}_c$, we will contradict the assumption that this set has size at most $\theta$. 

  To find this set of pairs $(\xi, i)$ we will repeatedly thin out the sets we have constructed to get more and more homogeneity. To begin this process, for each $\xi < \theta^+$ and $\alpha < \lambda$ we can pick $(p_{\xi, \alpha}, i_{\xi, \alpha})$ so that $i_{\xi,\alpha} < \mu$ and $p_{\xi, \alpha} \forces \check{\alpha} \in B_{\xi, i_{\xi, \alpha}}$. Note that any such $p_{\xi, \alpha}$ forces in particular that $\check{\alpha} \in \dot{A}_\xi$ and therefore we have that
 \begin{equation}\label{pxialpha}
\langle \theta^+\cdot\alpha + \xi, 1 \rangle \in p_{\xi, \alpha}.
 \end{equation}
Applying pigeonhole judiciously to this situation and further thinning out we can find moreover for each $\alpha < \lambda$ a set $u_\alpha \in [\theta^+]^{\theta^+}$ so that:
\begin{itemize}
\item the sequence $\vec{i}_\alpha = \seq{i_{\xi, \alpha} : \xi \in u_\alpha}$ is bounded in $\mu$;
\item $\langle \dom(p_{\xi, \alpha}) : \xi \in u_\alpha\rangle$ forms a $\Delta$-system;
\item $\langle p_{\xi,\alpha}:\xi\in u_\alpha\rangle$ are pairwise compatible with common intersection $r_\alpha$.
\end{itemize}
Fixing $\alpha$, note that at most finitely many $\xi \in u_\alpha$ can $\theta^+\cdot\alpha + \xi \in {\dom}(r_\alpha)$. By throwing out these finitely many elements we can assume that 
\begin{equation}\label{ralpha}    
\forall \alpha<\lambda \forall \xi \in u_\alpha\, [\theta^+\cdot\alpha+\xi\not\in{\dom}(r_\alpha)].
\end{equation}
  %for {\em no} $\xi \in u_\alpha$ is $\theta^+\cdot\alpha + \xi$ in the domain of $p_\alpha^*$. 

  \begin{claim}
      There are $X_* \in [\lambda]^\lambda$, $u_* \in [\theta^+]^{\theta^+}$ and $\vec{i}_* \in {}^{\theta^+}\mu$ so that for every $\alpha \in X_*$ we have that $u_\alpha = u_*$ and $\vec{i}_\alpha = \vec{i}_*$. 
  \end{claim}

  \begin{proof}[Proof of Claim]
      This is another pigeonhole argument. There are $(\theta^+)^{\theta^+}$ possibilities for $u_\alpha$ and $(\theta^+)^{\theta^+}<\mu<\lambda$ because $\mu$ is strong limit, so there exist $X\in [\lambda]^\lambda$ and $u_\ast\in [\theta^+]^{\theta^+}$ such that $u_\alpha=u_\ast$ for all $\alpha\in X$. Now, if $\alpha\in X$, $\vec{i}_\alpha$ is bounded in $\mu$, so we can find $\nu_\alpha<\mu$ such that $i_{\xi,\alpha}<\nu_\alpha$ for all $\xi\in u_*$. Shrinking $X$ if necessary, we may find $\nu<\mu$ such that $\nu_\alpha=\nu$ for all $\alpha\in X$. But now there are $\nu^{\theta^+}<\mu$ many possibilities for $\vec{i}_\alpha$, so we can now find $X_\ast\in [X]^\lambda$ and $\vec{i}_\ast$ such that $\vec{i}_\alpha=\vec{i}$ for all $\alpha\in X_\ast$.
  \end{proof}    
      %Note that we have $2^\theta <\mu^{<\mu} = \mu < \lambda = {\rm cf}(\lambda)$ there are less than $\lambda$ many possibilities for $(u_\alpha, i_\alpha)$ so $\lambda$-many of them must be the same.

 Next we aim for even more homogeneity. Let $S:=\{\alpha\in [\mu,\lambda):\cf(\alpha)=\theta^{++}\}$. For each $\delta \in S$, set $\beta_\delta:=\min(X_*\setminus \delta)$.

 \begin{claim}
     For each $\delta \in S$ there is a $\gamma_\delta \in [\mu,\delta)$ so that, if $\xi \in u_{*}$, then $p_{\xi, \beta_\delta} \hook \delta = p_{\xi, \beta_\delta}\hook \gamma_\delta$.
 \end{claim}

 \begin{proof}[Proof of Claim]
     By the finiteness of the domains and the fact that $\delta$ is a limit ordinal there is always a $\gamma_{\delta, \xi} < \delta$ with the property that $p_{\xi, \beta_\delta} \hook \delta = p_{\xi, \beta_\delta}\hook \gamma_{\delta, \xi}$. The point is that now since $\delta$ has cofinality $\theta^{++}$ and the $\xi$'s range over a subset of $\theta^+$ there is a $\gamma_\delta = {\sup}_{\xi \in u_{*}} \gamma_{\delta, \xi} < \delta$.
 \end{proof}
 
 Since $\delta \mapsto \gamma_\delta$ is a regressive function on $S$ we can find a stationary set $S_* \subseteq S$ with the property that for all $\delta \in S_*$ we have $\gamma_\delta = \gamma_*$ for some fixed $\gamma_*$. Note that $|\gamma_*| = \mu$. Now for each $\alpha \in S_*$ consider the function $\alpha \mapsto {\sup}[\bigcup \{{\dom}(p_{\xi, \beta_\alpha}) : \xi \in u_*\}]$. We can find a club $E$ bounding this function - i.e. for every $\alpha \in S_*$ we have
 \begin{equation}\label{Eclub}
{\sup}\left[\bigcup \{{\dom}(p_{\xi, \beta_\alpha}) : \xi \in u_*\}\right] < {\rm min}[ E \setminus (\alpha + 1)].
 \end{equation}
 Let $X_{**}$ be the set $\{\beta_\delta : \delta \in S_* \cap E\}$. %Let us take stock of where we are now. SOMETHING HERE ABOUT THE INTERSECTION OF SUPPORTS BOUNDED BY GAMMA*

 \begin{claim} \label{claimaboutsupports}
     Let $\delta,\epsilon\in S_\ast\cap E$ with $\delta<\epsilon$. If $\beta_\delta,\beta_\epsilon$ are distinct elements of $X_{**}$ then we have that \[\bigcup \{{\dom}(p_{\xi, \beta_\delta}) : \xi \in u_{*}\} \cap \bigcup \{{\dom}(p_{\xi, \beta_\epsilon}) : \xi \in u_{*}\} \subseteq \gamma^*.\] 
 \end{claim}

 \begin{proof}
    Fix $\xi,\eta\in u_\ast$. By \eqref{Eclub}, $\dom(p_{\xi\beta_\delta})\subseteq \epsilon$, hence
\begin{align*}
\dom(p_{\xi,\beta_\delta})\cap{\dom}(p_{\eta,\beta_\epsilon})& ={\dom}(p_{\xi,\beta_\delta})\cap \dom(p_{\eta,\beta_\epsilon})\cap \epsilon \\
& = \dom(p_{\xi,\beta_\delta})\cap \dom(p_{\xi,\beta_\epsilon}) \cap \gamma_*,\\
& \subseteq \gamma^*,
\end{align*}
where the second equality is by the choice of $\gamma_*$.
\end{proof}

     %By the last line before the statement of the claim we have that if $\alpha < \beta \in S_* \cap E$ and ${\dom}(p_{\xi, \beta_\alpha}) \subseteq \beta$ then for every $\xi' \in u_*$ we have ${\dom}(p_{\xi, \beta_\alpha}) \cap {\dom} (p_{\xi', \beta_\beta}) = {\dom}(p_{\xi, \beta_\alpha}) \cap \beta \cap {\dom} (p_{\xi', \beta_\beta}) \cap \beta$ (trivially) but by the choice of $\gamma^*$ then we have ${\dom} (p_{\xi', \beta_\beta}) \cap \beta = {\dom} (p_{\xi', \beta_\beta}) \cap \gamma^*$ so ${\dom}(p_{\xi, \beta_\alpha}) \cap {\dom} (p_{\xi', \beta_\beta}) = {\dom}(p_{\xi, \beta_\alpha}) \cap {\dom} (p_{\xi', \beta_\beta}) \cap \gamma^*$ which of course is contained in $\gamma^*$.

  Let $\langle W_n : n < \omega\rangle$ be a $\subseteq$-increasing sequence of sets of size ${<}\mu$ so that $\bigcup_{n < \omega} W_n = \gamma^*$. For $\delta \in E \cap S_*$ and $\xi \in u_*$ there is an $n_{\delta, \xi} \in \omega$ so that
  \begin{equation}\label{W}
      {\dom}(p_{\xi, \beta_\delta}) \cap \gamma_* \subseteq W_{n_{\delta, \xi}}
  \end{equation}
  simply by the finiteness of the domains. But this implies that for every $\delta \in S_* \cap E$ there is a $u_{*, \delta} \in [u_*]^{\theta^+}$ so that $n_{\delta, \xi} = n_\delta$ for some fixed $n_\delta \in \omega$ and every $\xi \in u_{*, \delta}$. By a pigeonhole argument, we can find a stationary set $S_{**} \subseteq S_{*} \cap E$, together with $n\in\omega$, $u_{**}\in [u_*]^{\theta^+}$ and a family $\seq{p_\xi:\xi\in u_{**}}$ so that, for every $\delta \in S_{**}$, the following hold: 
  \begin{itemize}
    \item $n_\delta = n$;
    \item $u_{*, \delta} = u_{**}$;
    \item $\seq{p_{\xi, \beta_\delta} \hook W_n : \xi \in u_{**}} = \seq{p_\xi : \xi \in u_{**}}$.
  \end{itemize}
   %for some fixed $n$ and  for some fixed $u_{**}$ and a family of conditions $\{p_\xi : \xi \in u_{**}\}$ so that $\seq{p_{\xi, \beta_\delta} \hook W_n : \xi \in u_{**}} = \seq{p_\xi : \xi \in u_{**}}$.

  Reaching back to the beginning of the proof we can observe that $\seq{p_\xi : \xi \in u_{**}}$ forms a $\Delta$-system by merit of being a family of restrictions of $\langle p_{\xi, \alpha} : \xi \in u_*\rangle$ which was constructed. Therefore, we can find a finite set $y \subseteq W_n$ and $p^*$ so that for distinct $\xi_1, \xi_2 \in u_{**}$ we have ${\dom}(p_{\xi_1}) \cap {\dom}(p_{\xi_2}) = y$ and $p_{\xi_1} \hook y = p_{\xi_2} \hook y = p^*$. Let $c \subseteq \{\beta_\delta : \delta \in S_{**}\}$ be of size $\sigma$. Let $\dot{x}_\xi$ be a name for $\dot{B}_{\xi, i_\xi} \cap \check{c}$. By the assumption on the $\dot{B}_{\xi, i}$'s and the formulation of Definition \ref{maindef} we must have that $\forces \dot{x}_\xi \in \dot{\overline{\F}}_{\check{c}}$ for every $\xi \in u_{**}$.

We will now contradict (B) by showing that:
  
  \begin{claim}
      $p^*$ forces that there is a set $u_{\dagger} \in [u_{**}]^{\theta^+}$ so that $\dot{x}_{\xi_1} \neq \dot{x}_{\xi_2}$ for every pair of distinct $\xi_1, \xi_2 \in u_{\dagger}$. 
  \end{claim}

\begin{proof}[Proof of Claim]
Let $\dot{u}_{\dagger}$ be the name for the set of $\xi \in u_{**}$ so that $p_{\xi} \in \dot{G}$. This set is forced by $p^*$ to have size $\theta^+$. To see this, observe that otherwise some $q \leq p^*$ forces it to have size at most $\theta$ so fix such a $q$ and let $\delta < \theta^+$ be an ordinal so that $q \forces \dot{u}_{\dagger} \subseteq u_{**} \cap \delta$. Fix an $\xi \in u_{**}$. Since $q$ strengthens $p^*$, if $p_\xi$ and $q$ are incompatible, there is an $\eta \notin y$ so that $p_\xi(\eta)\neq q(\eta)$ and in particular their domains intersect outside of $y$. Since the domains of the $p_\xi$'s however form a $\Delta$-system there is at most one such $\xi$ for $\eta \in \dom(q) \setminus y$. As $q$ has finite domain, there are cofinitely many $\xi\in u_{**}$ so that $p_\xi$ and $q$ are compatible, so in particular there is one such $\xi$ with $\xi>\delta$. But now if $r \leq p_\xi, q$ then $r \forces \xi \in \dot{u}_{\dagger} \setminus \delta$, which is absurd.

Now, towards a contradiction suppose that $\xi_1 \neq \xi_2 \in u_{**}$ and let $p' \leq p^*$ force that $\xi_1, \xi_2 \in \dot{u}_{\dagger}$ but $\dot x_{\xi_1} = \dot{x}_{\xi_2}$. By Claim \ref{claimaboutsupports} ,for each $\zeta\in\dom(p')\setminus \gamma_\ast$, there exists at most one $\beta\in c$ such that $\zeta\in \dom(p_{\xi_1,\beta})$. Since $c$ is infinite and $\dom(p')$ is finite, it follows that we can find some $\beta\in c$ such that
\begin{equation}\label{abovegamma}
\forall \zeta\in\dom(p')\setminus \gamma_\ast [\zeta\not\in\dom(p_{\xi_1,\beta})]
\end{equation}
and
\begin{equation}\label{outdomp}
\theta^+\cdot\beta+\xi_2\not\in\dom(p').
\end{equation}

%Observe that for any $\beta_\delta \in c$ if $\gamma \in \dom (p_{\xi_1, \beta_\delta}) \cap \dom(p')$ then either $\gamma \leq \gamma^*$ and hence $\gamma \in \dom p_{\xi_1}$ by the way $p_{\xi_1}$ was defined so therefore $p_{\xi_1, \beta_\delta}(\gamma)$ and $p'(\gamma)$ are compatible or $\gamma > \gamma^*$, hence $\beta_\delta$ is unique so that $\gamma \in \dom p_{\xi_1}$ by Claim \ref{claimaboutsupports}. Since $c$ is infinite but the domain of $p'$ is finite, there is only therefore a finite number of elements $\beta_\delta$ of $c$ so that $\dom (p_{\xi_1, \beta_\delta}) \cap \dom(p')$ contains an element outside of $y$. Let $c' \subseteq c$ be the cofinite set of elements not like this. 

%Now since $c'$ is infinite we can find moreover a $\beta_\delta \in c'$ so that $\theta^+\cdot\beta_\delta+\xi_2\not \in\dom(p')$. 

By \eqref{pxialpha},  recall that we have that $\theta^+\cdot\beta+\xi_2\in \dom(p_{\xi_2,\beta})$. On the other hand, by \eqref{ralpha},
\[
\theta^+\cdot\beta +\xi_2\not\in \dom(r_{\beta})=\dom(p_{\xi_1\beta})\cap \dom(p_{\xi_2\beta}),
\]
and therefore $\theta^+\cdot\beta+\xi_2\not\in\dom(p_{\xi_1\beta})$. Define
\[
r:= p' \cup [p_{\xi_1, \beta} \restr ({\dom}(p_{\xi_1, \beta}) \setminus {\dom}(p^*))] \cup \{\langle \theta^+\cdot \beta + \xi_2, 0\rangle \}
\]
We need to argue that $r$ is a condition, the issue being the compatibility of $p'$ and $p_{\xi_1,\beta}$. Fix $\zeta\in \dom(p')\cap\dom(p_{\xi_1,\beta})$. By \eqref{abovegamma}, $\zeta<\gamma^*$, hence $\zeta\in W_n$ by \eqref{W}, But $p_{\xi_1,\beta}\restr W_n=p_{\xi_1}$ and $p'\Vdash \xi_1\in u_{\dagger}$, so $p'$ and $p_{\xi_1}$ are compatible, therefore $p_{\xi_1,\beta}(\zeta)=p'(\zeta)$, as desired.

We have shown that $r$ is a condition extending $p_{\xi_1, \beta}$, so $r\Vdash \beta\in B_{\xi_1,i_{\xi_1}}$ by the choice of $p_{\xi_1,\beta}$. On the other hand, $r\Vdash \beta \notin \dot A_{\xi_2}$ because $\langle \theta^+\cdot \beta + \xi_2, 0\rangle$, and therefore $r\Vdash \beta \notin \dot{B}_{\xi_2, i}$ for {\em any} $i < \mu$. This contradicts the assumption that $p'$ forces $\dot x_{\xi_1} = \dot x_{\xi_2}$, thus proving the claim.
\end{proof}

By the last claim, $p^*$ forces that $\dot{\overline{\F}}_{\check{c}}$ has size at least $\theta^+$, so we have reached our final, promised, contradiction so we are done. 
\end{proof}
\begin{remark}
We note that in the above we really only needed that $\cf(
\mu) < \theta^+$ to run the argument. Under this assumption, one writes $\gamma_\ast=\bigcup_{j<\nu} W_j$ for some regular $\nu\le \theta$. The proof then proceeds the same, \textit{mutatis mutandis}.
\end{remark}

The principle $\mathsf{Pr}(\sigma, \theta, \mu, \lambda)$ is similar - though not evidently the same as that considered in \cite{Sh:108}, see also \cite{Sh:88a} for a more complete treatment. In \cite{Sh:108} principles of this form are shown to imply several mathematical statements involving e.g. almost free groups. It would be interesting to know whether the same holds here.
\begin{question}
What applications does $\mathsf{Pr}(\sigma, \theta, \mu, \lambda)$ have to other types of structures? Does a statement of this form distinguish adding $\aleph_\omega$ from adding $\aleph_{\omega+1}$ -many Cohen reals?
\end{question}

More generally it would be nice to know more about the principle $\mathsf{Pr}(\sigma, \theta, \mu, \lambda)$. On this note we ask the following.

\begin{question}
What consequences more generally does $\mathsf{Pr} (\sigma, \theta, \mu, \lambda)$ have for the choices of $\sigma$, $\theta$, $\mu$ and $\lambda$ considered in this paper? For example, could it be equivalent to some partition principle or its negation?
\end{question}

We note that our analysis above does not allow us to separate these principles when the final coordinate is different. Therefore we ask the following.

\begin{question}
Is it consistent that $\Pr(\sigma,\theta,\mu,\lambda_0) \wedge \neg \Pr(\sigma,\theta,\mu,\lambda_1)$ holds for some $\lambda_0<\lambda_1$.
\end{question}

Finally along these lines we also ask, 

\begin{question}
Can $\Pr$ exhibit any incompactness? For example, is it consistent that $\Pr(\aleph_0,\aleph_1,\aleph_2,\aleph_n)$ holds for every $n>2$ but $\Pr(\aleph_0,\aleph_1,\aleph_2,\aleph_\omega)$ fails?
\end{question}

\bibliographystyle{plain}
\bibliography{shlhetal,extra_refs}

\end{document}